\input ./style/arxiv-general.cfg
\documentclass[MSNbibl,number,citesort,seceqn,dvips]{arxbj}
\makeatletter
   \@ifpackageloaded{graphicx}{}{\usepackage{graphicx}}
\makeatother

\aid{0}
\volume{21}
\issue{4}
\pubyear{2015}
\firstpage{2308}
\lastpage{2335}
\doi{10.3150/14-BEJ645} 
\docsubty{FLA}

\makeatletter

\newcommand{\rrvert}{\vert}

\newcommand{\rrVert}{\Vert}
\newcommand{\llvert}{\vert}
\newcommand{\llVert}{\Vert}
\newtheorem{lemma}{Lemma}[section]
\newremark{remark}{Remark}[section]
\newtheorem{theorem}{Theorem}[section]
\newtheorem{corollary}{Corollary}[section]
\newtheorem{proposition}{Proposition}[section]
\newproclaim{definition}{Definition}[section]
\newproclaim{assumption}{Assumption}[section]

\newcommand{\fraca}[2]{{#1}/{#2}}
\newcommand{\overset}{\stackrel}
\newcommand{\m}{\mathcal}
\newcommand{\mb}{\mathbb}
\newcommand{\argmin}{\mathop{\operatorname{arg\,min}}}
\newcommand{\supp}{\operatorname{supp}}
\newcommand{\tr}{\operatorname{tr}}
\newcommand{\med}{\operatorname{med}}
\newcommand{\var}{\operatorname{Var}}
\newcommand{\proj}{\operatorname{Proj}}
\newcommand{\op}{{\operatorname{Op}}}
\newcommand{\rank}{\operatorname{rank}}
\newcommand{\eps}{\varepsilon}
\makeatother

\begin{document}
\begin{frontmatter}

\title{Geometric median and robust estimation in Banach spaces}
\runtitle{Geometric median and robust estimation}

\begin{aug}
\author[A]{\inits{S.}\fnms{Stanislav}~\snm{Minsker}\corref{}\ead[label=e1]{sminsker@math.duke.edu}}
\address[A]{Mathematics Department, Duke University, Box 90320,
Durham, NC 27708-0320, USA.\\ \printead{e1}}
\end{aug}

\received{\smonth{11} \syear{2013}}
\revised{\smonth{5} \syear{2014}}

%
\begin{abstract}
In many real-world applications, collected data are contaminated by
noise with heavy-tailed distribution and might contain outliers of
large magnitude.
In this situation, it is necessary to apply methods which produce
reliable outcomes even if the input contains corrupted measurements.
We describe a general method which allows one to obtain estimators with
tight concentration around the true parameter of interest taking values
in a Banach space.
Suggested construction relies on the fact that the geometric median of
a collection of independent ``weakly concentrated'' estimators
satisfies a much
stronger deviation bound than each individual element in the collection.
Our approach is illustrated through several examples,
including sparse linear regression and low-rank matrix recovery problems.
\end{abstract}

%
\end{frontmatter}

\section{Introduction}
\label{sec:intro}

Given an i.i.d. sample $X_1,\ldots,X_n\in\mb R$ from a distribution
$\Pi$ with $\var(X_1)<\infty$ and $t>0$, is it possible to construct
an estimator
$\hat\mu$ of the mean $\mu=\mb EX_1$ which would satisfy
%
%
\begin{equation}
\label{eq:question} \Pr \biggl(|\hat\mu-\mu|>C\sqrt{\var(X_1)
\frac{t}{n}} \biggr)\leq\mathrm{e}^{-t}
\end{equation}
for some absolute constant $C$ without \textit{any} extra assumptions
on $\Pi$?
What happens if the sample contains a fixed number of outliers of
arbitrary nature? Does the estimator still exist?

A (somewhat surprising) answer is yes, and several ways to construct
$\hat\mu$ are known.
The earliest reference that we are aware of is the book by
Nemirovski and Yudin \cite{Nemirovski1983Problem-complex00}, where
related question was investigated in the context of stochastic optimization.
We learned about problem (\ref{eq:question}) and its solution from the
work of Oliveira and Lerasle \cite{lerasle2011robust} who
used the ideas in spirit of \cite{Nemirovski1983Problem-complex00} to
develop the theory of ``robust empirical mean estimators''.
Method described in \cite{lerasle2011robust} consists of the following
steps: divide the given sample into $V\approx t$ blocks, compute the
sample mean within each block and then take the median\vadjust{\goodbreak} of these sample means.
A relatively simple analysis shows that the resulting estimator indeed
satisfies (\ref{eq:question}).
Similar idea was employed earlier in the work of Alon, Matias and
Szegedy \cite{alon1996space} to construct randomized algorithms for
approximating the so-called ``frequency moments'' of a sequence.
Recently, the aforementioned ``median of the means'' construction
appeared in \cite{bubeck2012bandits} in the context of multi-armed
bandit problem under weak assumptions on the reward distribution.
A different approach to the question (\ref{eq:question}) (based on
PAC-Bayesian truncation) was given in \cite{catoni2012challenging}.
A closely related independent recent work \cite{Hsu2013Loss-minimizati00}
applies original ideas of Nemirovski and
Yudin to general convex loss minimization.\looseness=-1

The main goal of this work is to design a general technique that allows
construction of estimators satisfying a suitable version of (\ref{eq:question})
for Banach space-valued $\mu$.
To achieve this goal, we show that a collection of independent
estimators of a Banach space-valued parameter can be transformed into a
new estimator which preserves the rate and admits much tighter
concentration bounds.
The method we propose is based on the properties of a
\textit{geometric median}, which is one of the possible extensions of a
univariate median to higher dimensions.

Many popular estimators (e.g., Lasso \cite{tibshirani1996regression}
in the context of high-dimensional linear regression) admit strong
theoretical guarantees if the distribution of the noise satisfies
restrictive assumptions (such as sub-Gaussian tails).
An important question that we attempt to answer is: can one design
algorithms which preserve nice properties of existing techniques and at
the same time:
\begin{enumerate}[(3)]
\item[(1)] admit strong performance guarantees under weak assumptions
on the noise;
\item[(2)] are not affected by the presence of a fixed number of
outliers of arbitrary nature and size;
\item[(3)] can be implemented in parallel for faster computation with
large data sets.
\end{enumerate}
Our results imply that in many important applications the answer is positive.
In Section~\ref{sec:examples}, we illustrate this assertion with
several classical examples, including principal component analysis,
sparse linear regression and low-rank matrix recovery.
In each case, we present non-asymptotic probabilistic bounds describing
performance of proposed methods.

For an overview of classical and modern results in robust statistics,
see \cite{Huber2009Robust-statisti00,hubert2008high}, and references therein.
Existing literature contains several approaches to estimation in the
presence of heavy-tailed noise based on the aforementioned estimators
satisfying (\ref{eq:question}).
However, most of the previous work concentrated on one-dimensional
versions of (\ref{eq:question}) and used it as a tool to solve
intrinsically high-dimensional problems.
For example, in \cite{lerasle2011robust} authors develop robust
estimator selection procedures based on the medians of empirical risks
with respect to disjoint subsets of the sample.
While this approach admits strong theoretical guarantees, it requires
several technical assumptions that are not always easy to check it practice.
Another related work \cite{audibert2011robust} discusses robust
estimation in the context of ridge regression.
Proposed method is based on a ``min--max'' estimator which has good
theoretical properties but can only be evaluated approximately based on
heuristic methods.
It is also not immediately clear if this technique can be extended to
robust estimation in other frameworks.
An exception is the approach described in \cite{Nemirovski1983Problem-complex00}
and further explored in \cite{Hsu2013Loss-minimizati00}, where authors
use a version of the
multidimensional median for estimator selection.
However, this method has several weaknesses in statistical applications
when compared to our technique; see Section~\ref{sec:main} for more
details and discussion.
The main results of our work require minimal assumptions, apply to a
wide range of models, and allow to use many existing algorithms as a
subroutine to produce robust estimators which can be evaluated exactly
via a simple iterative scheme.

\section{Geometric median}

Let $\mb X$ be a Banach space with norm $\|\cdot\|$, and let $\mu$ be
a probability measure on $(\mb X,\|\cdot\|)$.
Define the \textit{geometric median} (also called the spatial median,
Fermat--Weber point \cite{weber1929uber} or Haldane's median \cite
{haldane1948note}) of $\mu$ as
\[
x_\ast=\argmin_{y\in\mb X} \int_\mb X \bigl(\|y-x
\|-\|x\| \bigr)\mu(\mathrm{d}x).
\]
For other notions of the multidimensional median and a nice survey of
the topic, see \cite{small1990survey}.
In this paper, we will only be interested in a special case when $\mu$
is the empirical measure corresponding to a finite collection of points
$x_1,\ldots,x_k\in\mb X$, so that
%
%
\begin{equation}
\label{eq:median} x_\ast=\med(x_1,\ldots,x_k):=
\argmin_{y\in\mb X}\sum_{j=1}^k \|
y-x_j\|.
\end{equation}
Geometric median exists under rather general conditions.
For example, it is enough to assume that $\mb X=\mb Y^\ast$, where
$\mb Y$ is a separable Banach space and $\mb Y^\ast$ is its dual -- the
space of all continuous linear functionals on $\mb Y$.
This includes the case when $\mb X$ is separable and reflexive, that is,
$\mb X=  (\mb X^\ast  )^\ast$.
Moreover, if the Banach space $\mb X$ is strictly convex (i.e., $\|
x_1+x_2\|<\|x_1\|+\|x_2\|$ whenever $x_1$ and $x_2$ are not proportional),
then $x_\ast$ is unique unless all the points $x_1,\ldots,x_n$ are on
the same line.
For proofs of these results, see \cite{kemperman1987median}.
Throughout the paper, it will be assumed that $\mb X$ is separable and
reflexive.

In applications, we are often interested in the situation when $\mb X$
is a Hilbert space
(in particular, it is reflexive and strictly convex) and $\|\cdot\|$
is induced by the inner product
$ \langle\cdot,\cdot \rangle$.
In such cases, we will denote the ambient Hilbert space by $\mb H$.

The cornerstone of our subsequent presentation is the following lemma,
which states that if a given point $z$ is ``far'' from the geometric median
$x_\ast=\med(x_1,\ldots,x_k)$, then it is also ``far'' from a
constant fraction of the points $x_1,\ldots,x_k$.
We will denote $F(y):=\sum_{j=1}^k \|y-x_j\|$.
%

%
%
\begin{lemma}\label{lem:median}
\begin{enumerate}[(b)]
\item[(a)]Let $x_1,\ldots,x_k\in\mb H$ and let $x_\ast$ be their
geometric median.
Fix $\alpha\in  (0,\frac{1}2   )$ and assume that $z\in\mb
H$ is
such that $\|x_\ast-z\| > C_\alpha r$, where
%
%
\begin{equation}
\label{eq:c_alpha} C_\alpha= (1-\alpha)\sqrt{\frac{1}{1-2\alpha}}
\end{equation}
and $r>0$.
Then there exists a subset $J\subseteq  \{1,\ldots, k  \}$ of
cardinality $|J|>\alpha k$ such that for all $j\in J$, $\|x_j-z\|>r$.

\item[(b)]For general Banach spaces, the claim holds with a constant
$C_\alpha
=\frac{2(1-\alpha)}{1-2\alpha}$.
\end{enumerate}
\end{lemma}

\begin{pf}
(a)
Assume that the implication is not true.
Without loss of generality, it means that $\|x_i-z\|\leq r$,
$i=1,\ldots,
  \lfloor(1-\alpha)k  \rfloor+1$.

Consider the directional derivative
\[
DF(x_\ast;z-x_\ast):=\lim_{t\searrow0}
\frac{F(x_\ast+t(z-x_\ast
))-F(x_\ast)}{t}
\]
at the point $x_\ast$ in direction $z-x_\ast$.
Since $x_\ast$ minimizes $F$ over $\mb H$, $DF(x_\ast;z-x_\ast)\geq0$.
On the other hand, it is easy to see that
%
%
\begin{equation}
\label{eq:derivative1} \frac{DF(x_\ast;z-x_\ast)}{\|z-x_\ast\|}=-\sum_{j: x_j\ne x_\ast}
\frac{ \langle x_j-x_\ast,z-x_\ast \rangle}{\|x_j-x_\ast
\|\|z-x_\ast\|
}+\sum_{j=1}^k I
\{x_j=x_\ast\}.
\end{equation}
For $j=1,\ldots, \lfloor(1-\alpha)k\rfloor+1$ and $\gamma
_j=\arccos  (\frac{ \langle x_j-x_\ast,z-x_\ast
\rangle}{\|x_j-x_\ast\|\|
z-x_\ast\|}  ) $, we clearly have (see Figure~\ref{fig:01})
\[
\frac{ \langle x_j-x_\ast,z-x_\ast \rangle}{\|x_j-x_\ast
\|\|z-x_\ast\|
}=\cos(\gamma_j)> \sqrt{1-
\frac{1}{C_\alpha^2}},
\]
while $\frac{ \langle x_j-x_\ast,z-x_\ast \rangle}{\|
x_j-x_\ast\|\|z-x_\ast
\|}\geq-1$ for $j>\lfloor(1-\alpha)k\rfloor+1$.
This yields
\[
\frac{DF(x_\ast;z-x_\ast)}{\|z-x_\ast\|}< -(1-\alpha)k\sqrt {1-\frac{1}{C_\alpha^2}}+\alpha k
\leq0
\]
whenever $C_\alpha\geq(1-\alpha)\sqrt{\frac{1}{1-2\alpha}}$,
which leads to a contradiction.

\begin{figure}[t]

\includegraphics{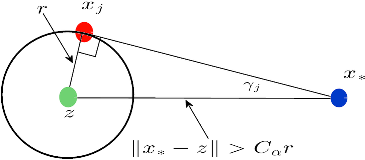}

\caption{Geometric illustration.}
\label{fig:01}
\end{figure}

(b)
See \hyperref[part2]{Appendix}.\vadjust{\goodbreak}
\end{pf}

%
\begin{remark}
\begin{enumerate}[(2)]
\item[(1)]Notice that in a Hilbert space, the geometric median $x_\ast
=\med
(x_1,\ldots,x_k)$ always belongs to the convex hull of $\{x_1,\ldots
,x_k\}$.
Indeed, if $x_\ast$ coincides with one of $x_j$'s,\vadjust{\goodbreak} there is nothing to prove.
Otherwise, since for any $v\in\mb H$ we have $DF(x_\ast;v)\geq0$ and
$DF(x_\ast;-v)\geq0$, it follows from (\ref{eq:derivative1}) that
$\sum_{j=1}^k \frac{x_j-x_\ast}{\|x_j-x_\ast\|}=0$,
which yields the result.

\item[(2)]In general Banach spaces, it might be convenient to consider
\[
\hat x_\ast:=\argmin_{y\in\operatorname{co}(x_1,\ldots,x_k)}\sum_{j=1}^k
\| y-x_j\|,
\]
where $\operatorname{co}(x_1,\ldots,x_k)$ is the convex hull of $\{
x_1,\ldots
,x_k\}$.
The claim of Lemma~\ref{lem:median} remains valid for $\hat x_\ast$
whenever $z\in\operatorname{co}(x_1,\ldots,x_k)$.
\end{enumerate}
\end{remark}

\section{``Boosting the confidence'' by taking the geometric median of
independent estimators}\label{sec:main}

A useful property of the geometric median is that it transforms a
collection of independent estimators that are ``weakly'' concentrated
around the true parameter of interest into a single estimator which
admits significantly tighter deviation bounds.
For $0<p<\alpha<\frac{1} 2$, define
\[
\psi(\alpha;p)=(1-\alpha)\log\frac{1-\alpha}{1-p}+\alpha\log \frac{\alpha}{p}.
\]
%

\begin{theorem}
\label{th:main}
Assume that $\mu\in\mb X$ is a parameter of interest, and let $\hat
\mu_1,\ldots,\hat\mu_k\in\mb X$ be a collection of independent
estimators of $\mu$.
Fix $\alpha\in  (0,\frac{1}{2}  )$. Let $0<p<\alpha$ and
$\eps
>0$ be such that for all~$j$, $1\leq j\leq k$,
%
%
\begin{equation}
\label{weak_con} \Pr \bigl(\|\hat\mu_j-\mu\|>\eps \bigr)\leq p.
\end{equation}
Set
%
%
\begin{equation}
\label{def:geom_med} \hat\mu:=\med(\hat\mu_1,\ldots,\hat
\mu_k).
\end{equation}
Then
%
%
\begin{equation}
\label{strong_con} \Pr \bigl(\|\hat\mu-\mu\|> C_\alpha\eps \bigr)\leq\mathrm
{e}^{-k\psi
(\alpha;p)},
\end{equation}
where $C_\alpha$ is a constant defined in Lemma~\ref{lem:median} above.
\end{theorem}

\begin{pf}
Assume that event $\m E:=  \{\|\hat\mu-\mu\|>C_\alpha\eps
  \}$ occurs.
Lemma~\ref{lem:median} implies that there exists a subset $J\subseteq
\{1,\ldots,k\}$ of cardinality $|J|\geq\alpha k$ such that
$\|\mu_j-\mu\|>\eps$ for all $j\in J$, hence
\[
\Pr(\m E)\leq\Pr \Biggl(\sum_{j=1}^k I \bigl \{
\|\hat\mu_j-\mu\| >\eps \bigr \}>\alpha k \Biggr).
\]
If $W$ has Binomial distribution $W\sim B(k,p)$, then
\[
\Pr \Biggl(\sum_{j=1}^k I \bigl \{\|\hat
\mu_j-\mu\|>\eps \bigr \} >\alpha k \Biggr)\leq\Pr (W>\alpha k )
\]
(see Lemma~23 in \cite{lerasle2011robust} for a rigorous proof of this fact).
Chernoff bound (e.g., Proposition A.6.1 in \cite
{Vaart1996Weak-convergenc00}) implies that
\[
\Pr (W>\alpha k )\leq\exp \bigl(-k\psi (\alpha ;p ) \bigr).
\]
\upqed
\end{pf}

%
\begin{remark}
\begin{enumerate}[(b)]
\item[(a)]If (\ref{weak_con}) is replaced by a weaker condition
assuming that
\[
\Pr \bigl(\|\hat\mu_j-\mu\|>\eps \bigr)\leq p<\alpha
\]
is satisfied only for $\hat\mu_j$, $j\in J\subset\{1,\ldots,k\}$,
where $|J|=(1-\tau)k$ for $0\leq\tau<\frac{\alpha-p}{1-p}$, then
the previous argument implies
\[
\Pr \bigl(\|\hat\mu-\mu\|> C_\alpha\eps \bigr)\leq\exp \biggl(-k(1-\tau)\psi
\biggl(\frac{\alpha-\tau}{1-\tau},p \biggr) \biggr).
\]
In particular, this version is useful in addressing the situation when
the sample contains a subset of cardinality at most $\tau k$ consisting
of ``outliers'' of arbitrary nature.

\item[(b)]
It is also clear that results of Theorem~\ref{th:main} can be used to
positively answer question (3) posed in the \hyperref[sec:intro]{Introduction}.
Indeed, if several autonomous computational resources (e.g.,
processors) are available, one can evaluate estimators $\hat\mu_j$,
$j=1,\ldots,k$ in parallel and combine them via the geometric median as
a final step. In many situations, the improvement in computational cost
will be significant.
\end{enumerate}
\end{remark}

Note that it is often easy to obtain an estimator satisfying
(\ref{weak_con}) with the correct rate $\eps$ under minimal
assumptions on
the underlying distribution.
In particular, if $\mu$ is the mean and $\hat\mu_k$ is the sample
mean, then (\ref{weak_con}) can be deduced from Chebyshev's
inequality, see Section~\ref{sec:mean} below for more details.

Next, we describe the method proposed in
\cite{Nemirovski1983Problem-complex00} which is based on a different notion
of the median.
Let $\hat\mu_1,\ldots,\hat\mu_k$ be a collection of independent
estimators of $\mu$ and assume that $\eps>0$ is chosen to satisfy
%
%
\begin{equation}
\label{eq:nemir2} \Pr \bigl(\|\hat\mu_j-\mu\|>\eps \bigr)\leq p<
\tfrac
{1}{2},\qquad1\leq j\leq k.
\end{equation}
Define $\tilde\mu:=\hat\mu_{j_\ast}$, where
%
%
\begin{eqnarray}
\label{eq:nemir0} &&j_\ast=j_\ast(\eps):=\min \biggl\{ j\in\{1,
\ldots, k\}: \exists I\subset \{1,\ldots, k\} 
\nonumber
\\[-8pt]
\\[-8pt]
&&\phantom{j_\ast=j_\ast(\eps):=\min \biggl\{}
\mbox{such
that } |I |>\frac{k} 2 \mbox{ and } \forall i\in I, \|\hat
\mu_i-\hat\mu_j\|\leq2\eps \biggr\},\nonumber
\end{eqnarray}
and $j_\ast=1$ if none of $\hat\mu_j$'s satisfy the condition in braces.
It is not hard to show that
%
%
\begin{equation}
\label{eq:nemir} \Pr \bigl(\|\tilde\mu-\mu\|>3\eps \bigr)\leq\mathrm {e}^{-k\psi(1/2;p)},
\end{equation}
which is similar to (\ref{strong_con}).

However, it is important to note that $\tilde\mu$ defined by
(\ref{eq:nemir0}) explicitly depends on $\eps$ which is often unknown
in practice,
while the ``geometric median'' estimator $\hat\mu$
(\ref{def:geom_med}) does not require any additional information.
%

\begin{remark}
It is possible to modify $\tilde\mu$ by choosing $\eps_\ast$ to be
the smallest $\eps>0$ for which (\ref{eq:nemir0}) defines a non-empty
set, and setting
$j_\ast:=j_\ast(\eps_\ast)$.
The resulting construction does not assume that $\eps$ satisfying
condition (\ref{eq:nemir2}) is known a priori, while (\ref{eq:nemir})
remains valid. See \cite{Hsu2013Loss-minimizati00} for discussion and
applications of this method.
\end{remark}

It is important to mention the fact that (\ref{eq:nemir}) and the
inequality (\ref{strong_con}) of Theorem~\ref{th:main} have different
constants in front of $\eps$: it is equal to $C_\alpha$ in
(\ref{strong_con}) and to $3$ in (\ref{eq:nemir}).
Note that in the Hilbert space case, $C_\alpha=(1-\alpha)\sqrt{\frac
{1}{1-2\alpha}}\to1$
as $\alpha\to0$, while for general Banach spaces $C_\alpha=\frac
{2(1-\alpha)}{1-2\alpha}\to2$ as $\alpha\to0$.
In particular, $C_\alpha<3$ for all sufficiently small $\alpha$
(e.g., for $\alpha<-8+6\sqrt{2}\approx0.485$ in Hilbert space framework).
This difference becomes substantial when $\eps$ is of the form
\[
\eps=\mbox{approximation error}+\mbox{random error},
\]
where the first term in the sum is a constant and the second term
decreases with the growth of the sample size.
This is a typical situation when the model is misspecified,
see Section~\ref{sec:low_rank} below for a concrete example related to
matrix regression.
Our method allows to keep the constant in front of the approximation
error term arbitrary close to $1$ (and often leads to noticeably better
constants in general).

\section{Examples}
\label{sec:examples}

In this section, we discuss applications of Theorem~\ref{th:main} to
several classical problems, namely, estimation of the mean, principal
component analysis, sparse linear regression and low-rank matrix recovery.

Our priority was simplicity and clarity of exposition of the main ideas
which could affect optimality of some constants and generality of
obtained results.

\subsection{Estimation of the mean in a Hilbert space}
\label{sec:mean}

Assume that $\mb H$ is a separable Hilbert space with norm $\|\cdot\|$.
Let $X,X_1,\ldots,X_n\in\mb H$, $n\geq2$, be an i.i.d. sample from a
distribution $\Pi$ such that
$\mb EX=\mu$, $\mb E  [(X-\mu) \otimes(X-\mu)  ]=\Sigma
$ is the
covariance operator, and
$\mb E\|X-\mu\|^2=\tr(\Sigma)<\infty$.
We will apply result of Theorem~\ref{th:main} to construct a ``robust
estimator'' of $\mu$.
Let us point out that a simple alternative estimator of $\mu$ can be
obtained by applying the univariate ``median of the means''
construction (explained in Section~\ref{sec:intro}) coordinatewise.
When $\dim(\mb H)<\infty$, this method leads to dimension-dependent
bounds that can be even better than the result for the geometric
median-based approach (when $\dim(\mb H)$ is small).
However, when $\dim(\mb H)$ is large or infinite, dimension-dependent
estimates become uninformative; see Remark~\ref{rmk:comparison} below
for more details.

Set $\alpha_\ast:=\frac{7}{18}$ and $p_\ast:=0.1$ (these numerical
values allow to optimize the constants in Corollary~\ref{th:mean} below).
Let $0<\delta<1$ be the confidence parameter, and set
\[
k=: \biggl\lfloor\frac{\log (\fraca{1}{\delta} )}{\psi
(\alpha_\ast;p_\ast)} \biggr\rfloor+1 \leq \biggl\lfloor3.5\log
\biggl(\frac{1}{\delta} \biggr) \biggr\rfloor+1
\]
(we will assume that $\delta$ is such that $k\leq\frac{n}{2}$).
Divide the sample $X_1,\ldots, X_n$ into $k$ disjoint groups
$G_1,\ldots, G_k$ of size $ \lfloor\frac{n}{k} \rfloor$
each, and define
%
%
\begin{eqnarray}
\label{eq:median_mean} \hat\mu_j&:=&\frac{1}{|G_j|}\sum
_{i\in G_j}X_i,\qquad j=1,\ldots, k,
\nonumber
\\[-8pt]
\\[-8pt]
\hat\mu&:=&\med(\hat\mu_1,\ldots,\hat\mu_k).
\nonumber
\end{eqnarray}

%
\begin{corollary}\label{th:mean}
Under the aforementioned assumptions,
%
%
\begin{equation}
\label{eq:deviation1} \Pr \biggl(\|\hat\mu-\mu\|\geq11\sqrt{\frac{\tr(\Sigma)\log
(1.4/\delta)}{n}} \biggr)
\leq\delta.
\end{equation}
\end{corollary}

\begin{pf}
We will apply Theorem~\ref{th:main} to the independent estimators
$\hat\mu_1,\ldots,\hat\mu_k$.

To this end, we need to find $\eps$ satisfying (\ref{weak_con}).
Since for all $1\leq j\leq k\leq\frac{n} 2$
\[
\mb E\|\hat\mu_j-\mu\|^2\leq\frac{\mb E\|X-\mu\|^2}{|G_j|}\leq
\frac{2k}{n}\tr(\Sigma),
\]
Chebyshev's inequality gives
\[
\Pr \bigl(\|\hat\mu_j-\mu\|\geq\eps \bigr)\leq\frac
{2k}{n\eps^2}\tr (\Sigma),
\]
which is further bounded by $p_\ast$ whenever $\eps^2\geq\frac
{2k}{p_\ast n}\tr(\Sigma)$.
The claim now follows from Theorem~\ref{th:main} and the bounds
$C_{\alpha_\ast}\sqrt{\frac{2}{p_\ast\psi(\alpha_\ast;p_\ast
)}}\leq11$ and $\log(1/\delta)+\psi(\alpha_\ast;p_\ast)\leq\log
 (\frac{1.4}{\delta} )$.
\end{pf}

%
\begin{remark}
\label{rmk:comparison}
\begin{enumerate}[(b)]
\item[(a)]It is easy to see that the proof of Corollary~\ref{th:mean} actually
yields a better bound
%
%
\begin{equation}
\label{eq:geom_opt} \Pr \biggl(\|\hat\mu-\mu\|\geq\frac{C_{\alpha_\ast}}{\sqrt
{p_\ast\psi(\alpha_\ast;p_\ast)}}\sqrt{\tr(\Sigma)
\frac{\log(1.4/\delta)}{n-3.5\log(1.4/\delta)}} \biggr)\leq\delta,
\end{equation}
with $\frac{C_{\alpha_\ast}}{\sqrt{p_\ast\psi(\alpha_\ast
;p_\ast)}}\leq7.6$.

\item[(b)]For the estimator $\tilde\mu$ defined in (\ref
{eq:nemir0}), it
follows from (\ref{eq:nemir}) (with $p=0.12$) that
\[
\Pr \biggl(\|\tilde\mu-\mu\|\geq13.2\sqrt{\tr(\Sigma)\frac{\log
(1.6/\delta)}{n-2.4\log(1.6/\delta)}} \biggr)
\leq\delta,
\]
which yields a noticeably larger constant ($13.2$ versus $7.6$).

\item[(c)]When $\mb H$ is a $D$-dimensional Euclidean space, it is
interesting to
compare $\hat\mu$ with another natural estimator $\hat\mu_{\ast}$
obtained by taking the median \emph{coordinate-wise}, that is, if
$\hat\mu_j=(\hat\mu_j^{(1)},\ldots,\hat\mu_j^{(D)})$,
$j=1,\ldots,
k$, then
\[
\hat\mu_\ast:= \bigl(\med\bigl(\hat\mu_1^{(1)},
\ldots,\hat\mu _k^{(1)}\bigr),\ldots,\med\bigl(\hat
\mu_1^{(D)},\ldots,\hat\mu _k^{(D)}
\bigr) \bigr).
\]
It is easy to see that for the univariate median, inequality
(\ref{strong_con}) holds with $\alpha=1/2$ and $C_\alpha=1$, hence the
union bound over $i=1,\ldots, D$ implies that
%
%
\begin{equation}
\label{eq:univar} \Pr \biggl(\|\hat\mu_\ast-\mu\|\geq4.4\sqrt{\tr(\Sigma)
\frac
{\log(1.6D/\delta)}{n-2.4\log(1.6D/\delta)}} \biggr) \leq \delta
\end{equation}
(here, $p$ was set to be $0.12$).
This bound should be compared to (\ref{eq:geom_opt}) -- the latter
becomes better only when $D$ is sufficiently large (e.g., $D\geq165$ for
$\delta=0.1$ and $D\geq15\,806$ for $\delta=0.01$).\footnote{We
want to thank the anonymous reviewer for pointing this out.}
Note that the constant in (\ref{eq:univar}) can be further improved in
a situation when tight upper bounds on the true variances or kurtoses
of coordinates of $X$ are known by using a univariate estimator of
\cite{catoni2012challenging} to construct $\hat\mu_\ast$.
\end{enumerate}
\end{remark}

Our estimation technique naturally extends to the problem of
constructing the confidence sets for the mean.
Indeed, when faced with the task of obtaining the non-asymptotic
confidence interval, one usually fixes the desired coverage probability
in advance, which is exactly how we build our estimator.
To obtain a parameter-free confidence ball from (\ref{eq:deviation1}),
one has to estimate $\tr(\Sigma)$.
To this end, we will apply Theorem~\ref{th:main} to a collection of
independent statistics $\hat T_1,\ldots, \hat T_k$, where
\[
\hat T_j=\frac{1}{|G_j|}\sum_{i\in G_j}
\|X_i-\hat\mu_j\|^2,\qquad j=1,\ldots, k,
\]
and $\hat\mu_j$ are the sample means defined in (\ref{eq:median_mean}).
Let $\hat T:=\med(\hat T_1,\ldots,\hat T_k)$ (if $k$ is even, the
median is not unique, so we pick an arbitrary representative).
%

\begin{proposition}
Assume that
%
%
\begin{equation}
\label{eq:var} 15.2\sqrt{\frac{\mb E\|X-\mu\|^4-  (\tr(\Sigma)
)^2}{  (\tr
(\Sigma)  )^2}}\leq \biggl(\frac{1}{2} -178
\frac{\log
(1.4/\delta
)}{n} \biggr)\sqrt{\frac{n}{\log(1.4/\delta)}}.
\end{equation}
Then
%
%
\begin{equation}
\label{eq:variance} \Pr \bigl(\tr(\Sigma)\leq2\hat T \bigr)\geq1-\delta.
\end{equation}
\end{proposition}

\begin{pf}
Note that
\[
\hat T_j=\frac{1}{|G_j|}\sum_{i\in G_j}
\|X_i-\mu\|^2-\|\hat\mu _j-\mu
\|^2.
\]
Chebyshev's inequality gives (assuming that $k\leq n/2$)
\begin{eqnarray*}
&\Pr \biggl(\|\hat\mu_j-\mu\|^2\geq\underbrace{4
\displaystyle\frac{\tr(\Sigma
)\log(1.4/\delta)}{p_\ast\psi(\alpha_\ast;p_\ast) n}}_{\eps
_1} \biggr)\leq\displaystyle\frac{p_\ast}{2},&
\\
&\Pr \biggl(\biggl\llvert \displaystyle\frac{1}{|G_j|}\sum_{i\in G_j}
 \|X_i-\mu \|^2-\tr(\Sigma)
\biggr\rrvert \geq
\underbrace{2\sqrt{
\mb E\|X-\mu\|^4- \bigl(\tr(\Sigma) \bigr)^2}\sqrt {
\displaystyle\frac{\log(1.4/\delta)}{np_\ast\psi(\alpha_\ast;p_\ast
)}}}_{\eps_2} \biggr)\leq\displaystyle\frac{p_\ast}{2},&
\end{eqnarray*}
hence $\Pr  (|\hat T_j-\tr(\Sigma)|\geq\eps_1+\eps_2
)\leq
p_\ast$.
Theorem~\ref{th:main} implies that
\[
\Pr \bigl(\hat T\leq\tr(\Sigma)-C_{\alpha_\ast}(\eps_1+\eps
_2) \bigr)\leq\Pr \bigl(\bigl |\hat T-\tr(\Sigma)\bigr |\geq C_{\alpha_\ast}(
\eps _1+\eps _2) \bigr)\leq\delta.
\]
Since $\Pr  (\hat T\leq\frac{\tr(\Sigma)}{2}  )\leq\Pr
  (\hat
T\leq\tr(\Sigma)-C_{\alpha_\ast}(\eps_1+\eps_2)  )$ whenever
(\ref{eq:var}) is satisfied,
the result follows.
\end{pf}

Combining (\ref{eq:variance}) with Corollary~\ref{th:mean}, we
immediately get the following statement.

%
\begin{corollary}
\label{th:conf_ball}
Let $B(h,r)$ be the ball of radius $r$ centered at $h\in\mb H$.
Define the random radius
\[
r_n:=11\sqrt2\sqrt{\widehat T\frac{\log(1.4/\delta)}{n}}
\]
and let $\hat\mu$ be the estimator defined by (\ref{eq:median_mean}).
If (\ref{eq:var}) holds, then
\[
\Pr \bigl(B(\hat\mu,r_n)\mbox{ contains } \mu \bigr)\geq1-2\delta.
\]
\end{corollary}

\subsection{Robust Principal Component Analysis}
\label{sec:pca}

It is well known that classical Principal Component Analysis (PCA)
\cite{Pearson1901On-lines-and-pl00} is very sensitive to the presence
of the outliers in a sample.
The literature on robust PCA suggests several computationally efficient
and theoretically sound methods to recover the linear structure in the data.
For instance, if part of the observations is contained in a
low-dimensional subspace while the rest are corrupted by noise, the
low-dimensional subspace can often be recovered exactly, see \cite
{candes2011robust,zhang2011novel} and references therein.

However, for the case when no additional geometric structure in the
data can be assumed, we suggest a simple and easy-to-implement
alternative which uses the geometric median to obtain a robust
estimator of the covariance matrix.
In this section, we study the simplest case when the geometric median
is combined with the sample covariance estimator.
However, it is possible to use various alternatives in place of the
sample covariance, such as the shrinkage estimator
\cite{ledoit2012nonlinear}, banding/tapering estimator \cite
{bickel2008regularized},
hard thresholding estimator \cite{bickel2008covariance} or the nuclear
norm-penalized
estimator \cite{lounici2012high}, to name a few.

Let $X,X_1,\ldots,X_n\in\mb R^D$ be i.i.d. random vectors such that
$\mb EX=\mu$, $\mb E  [(X-\mu)(X-\mu)^T  ]=\Sigma$ and
$\mb E\|X\|
^4<\infty$, where $\|\cdot\|$ is the usual Euclidean norm.
We are interested in estimating the covariance matrix $\Sigma$ and the
linear subspace generated by its eigenvectors associated to ``large''
eigenvalues.
For simplicity, suppose that all positive eigenvalues of $\Sigma$ have
algebraic multiplicity~1.
We will enumerate $\lambda_i:=\lambda_i(\Sigma)$ in the decreasing
order, so that $\lambda_1>\lambda_2\geq\cdots\geq0$.

Assume first that the data is centered (so that $\mu=0$).
As before, set $\alpha_\ast:=\frac{7}{18}$, $p_\ast:=0.1$, divide
the sample $X_1,\ldots,X_n$ into
$k= \lfloor\frac{\log (\fraca{1}{\delta} )}{\psi
(\alpha_\ast;p_\ast)} \rfloor+1$
disjoint groups $G_1,\ldots, G_k$ of size $ \lfloor\frac
{n}{k} \rfloor$ each,
and let
%
%
\begin{eqnarray}
\label{eq:median_cov} \hat\Sigma_j&:=&\frac{1}{|G_j|}\sum
_{i\in G_j}X_i X_i^T,\qquad
j=1,\ldots, k,
\nonumber
\\[-8pt]
\\[-8pt]
\hat\Sigma&:=&\med(\hat\Sigma_1,\ldots,\hat\Sigma_k),
\nonumber
\end{eqnarray}
where the median is taken with respect to Frobenius norm $\|A\|
_{{\mathrm
F}}:=\sqrt{\tr(A^T A)}$.

%
\begin{remark}
Note that $\hat\Sigma$ is positive semidefinite as a convex
combination of positive semidefinite matrices.
\end{remark}

Let $\proj_m$ be the orthogonal projector on a subspace corresponding
to the $m$ largest positive eigenvalues of $\Sigma$.
Let $\widehat{\proj_m}$ be the orthogonal projector of the same rank
as $\proj_m$ corresponding to the $m\leq \lfloor\frac{n}{k}
\rfloor$ largest eigenvalues of $\hat\Sigma$.
In this case, the following bound holds.
%

\begin{corollary}
Let $\Delta_m:=\lambda_m-\lambda_{m+1}$ and assume that
%
%
\begin{equation}
\label{eq:gap} \Delta_m > 44\sqrt{\frac{  (\mb E\|X\|^4-\tr(\Sigma^2)
)\log
(1.4/\delta)}{n}}.
\end{equation}
Then
\[
\Pr \biggl( \|\widehat{\proj_m}-\proj_m
\|_{{\mathrm F}}\geq \frac{22}{\Delta_m}\sqrt{\frac{  (\mb E\|X\|^4-\tr(\Sigma
^2)
)\log(1.4/\delta)}{n}} \biggr)\leq
\delta.
\]
\end{corollary}

\begin{pf}
It follows from Davis--Kahan perturbation theorem \cite
{davis1970rotation} (see also Theorem~3 in \cite
{Zwald2006On-the-Converge00}) that, whenever $\|\hat\Sigma-\Sigma\|
_{{\mathrm F}}<\frac{1} 4 \Delta_m$,
%
%
\begin{equation}
\label{eq:proj} \|\widehat{\proj_m}-\proj_m
\|_{{\mathrm F}}\leq\frac{2\|
\hat
\Sigma-\Sigma\|_{{\mathrm F}}}{\Delta_m}.
\end{equation}
Define $Y_j:=X_jX_j^T$, $j=1,\ldots, n$ and note that $\mb E\|Y-\mb
EY\|
^2_{{\mathrm F}}=\mb E\|X\|^4-\tr(\Sigma^2)$.
Applying Corollary~\ref{th:mean} to $Y_j$, $j=1,\ldots,n$, we get
\[
\Pr \biggl(\|\hat\Sigma-\Sigma\|_{{\mathrm F}}\geq11\sqrt{\frac
{  (\mb
E\|X\|^4-\tr(\Sigma^2)  )\log(1.4/\delta)}{n}}
\biggr)\leq \delta.
\]
Whenever (\ref{eq:gap}) is satisfied, inequality $11\sqrt{\frac
{
(\mb E\|X\|^4-\tr(\Sigma^2)  )\log(1.4/\delta)}{n}} < \frac
{\Delta_m}{4}$ holds, and (\ref{eq:proj}) yields the result.
\end{pf}

Similar bounds can be obtained in a more general situation when $X$ is
not necessarily centered.
To this end, let
%
%
\begin{eqnarray}
\label{eq:geom_cov} \hat\mu_j&:=&\frac{1}{|G_j|}\sum
_{i\in G_j}X_i,\qquad j=1,\ldots, k,
\nonumber
\\
\hat\Sigma_j&:=&\frac{1}{|G_j|}\sum
_{i\in G_j}(X_i-\hat\mu _j)
(X_i-\hat\mu_j)^T,\qquad j=1,\ldots, k,
\\
\hat\Sigma&:=&\med(\hat\Sigma_1,\ldots,\hat\Sigma_k).
\nonumber
\end{eqnarray}
Note that $\hat\Sigma_1,\ldots,\hat\Sigma_k$ are independent.
Then, using the fact that for any $1\leq j\leq k$
\[
\hat\Sigma_j=\frac{1}{|G_j|}\sum_{i\in G_j}(X_i-
\mu) (X_i-\mu)^T- (\mu-\hat\mu_j) (\mu-\hat
\mu_j)^T,
\]
it is easy to prove the following bound.

%
\begin{corollary}
Let
\[
\eps_n(\delta):=15.2\sqrt{\frac{  (\mb E\|X-\mu\|^4-\tr
(\Sigma
^2)  )\log(1.4/\delta)}{n}}+178\frac{\tr(\Sigma)\log
(1.4/\delta)}{n}
\]
and assume that $\Delta_m>4\eps_n(\delta)$.
Then
\[
\Pr \biggl(\llVert \widehat\proj_m-\proj_m\rrVert
_{{\mathrm
F}}\geq\frac{2\eps
_n(\delta)}{\Delta_m} \biggr)\leq\delta.
\]
\end{corollary}

\subsection{High-dimensional sparse linear regression}

Everywhere in this subsection, $\|\cdot\|$ stands for the standard
Euclidean norm, $\|\cdot\|_1$ denotes the $\ell_1$-norm and $\|
\cdot\|_\infty$ -- the sup-norm of a vector.

Let $x_1,\ldots,x_n\in\mb R^D$ be a fixed collection of vectors and
let $Y_j$ be noisy linear measurements of $\lambda_0\in\mb R^D$:
%
%
\begin{equation}
\label{eq:lin_model} Y_j=\lambda_0^T
x_j+\xi_j,
\end{equation}
where $\xi_j$ are independent zero-mean random variables such that
$\var(\xi_j)\leq\sigma^2, 1\leq j\leq n$.
Set $\mb X:=(x_1|\ldots|x_n)^T$.

We are interested in the case when $D\gg n$ and $\lambda_0$ is sparse,
meaning that
\[
N(\lambda_0):= \bigl |\supp(\lambda_0) \bigr |= \bigl |\{j: \lambda
_{0,j}\ne0\} \bigr |=s\ll D.
\]
In this situation, a (version of) the famous Lasso estimator \cite
{tibshirani1996regression} of $\lambda_0$ is obtained as a solution of
the following optimization problem:
%
%
\begin{equation}
\label{eq:lasso} \hat\lambda_\eps:= \argmin_{\lambda\in\mb R^D} \Biggl[
\frac{1}{n}\sum_{j=1}^n
\bigl(Y_j-\lambda^T x_j
\bigr)^2+\eps\| \lambda \|_1 \Biggr].
\end{equation}
The goal of this section is to extend the applicability of some
well-known results for this estimator to the case of a heavy-tailed
noise distribution.

Existing literature on high-dimensional linear regression suggests
several ways to handle corrupted measurements, for instance, by using a
different loss function (e.g., the so-called \textit{Huber's} loss
\cite{lambert2011robust}), or by implementing a more flexible penalty
term \cite{wright2010dense,Nguyen2013Robust-Lasso-Wi00}.
In particular, in \cite{Nguyen2013Robust-Lasso-Wi00} authors study the model
%
%
\begin{equation}
\label{eq:big_noise} Y=X\lambda_0+e^\ast+\xi,
\end{equation}
where $X\in\mb R^{n\times D}$, $\xi\in\mb R^n$ is the additive noise
and $e^\ast\in\mb R^n$ is a sparse error vector with unknown support
and arbitrary large entries.
It is shown that if the rows of $X$ are independent Gaussian random
vectors, then is possible to accurately estimate both $\lambda_0$ and
$e^\ast$ by adding an extra penalty term:
\[
(\tilde\lambda_\eps, \tilde e_\eps):= \argmin
_{\lambda\in\mb R^D, e\in\mb R^n} \biggl[\frac{1}{n}\llVert Y-X\lambda -e\rrVert
^2+\eps_1\|\lambda\|_1+\eps_2\|e
\|_1 \biggr].
\]
However, as in the case of the usual Lasso, confidence of estimation
depends on the distribution of $\xi$.
In particular, Gaussian-type concentration holds only if the entries of
$\xi$ have sub-Gaussian tails.

The main result of this subsection (stated in Theorem~\ref{th:lasso2})
provides strong performance guarantees for the robust version of the
usual Lasso estimator (\ref{eq:lasso}) and requires only standard
conditions on the degree of sparsity and restricted eigenvalues of the design.
Similar method can be used to improve performance guarantees for the
model (\ref{eq:big_noise}) in the case of heavy-tailed noise $\xi$.

Probabilistic bounds for the error $\|\hat\lambda_\eps-\lambda_0\|$
crucially depend on integrability properties of the noise variable.
We will recall some known bounds for the case when $\xi_j\sim
N(0,\sigma^2)$, $j=1,\ldots, n$ (of course, similar results hold for
\textit{sub-Gaussian} noise as well).
For $J\subset\{1,\ldots,D\}$ and $u\in\mb R^D$, define $u_J\in\mb
R^D$ by $  (u_J  )_j=u_j, j\in J$ and $  (u_J
)_j=0, j\in J^c$
(here, $J^c$ denotes the complement of a set $J$).

%
\begin{definition}[(Restricted eigenvalue condition, \cite
{bickel2009simultaneous})]
Let $1\leq s\leq D$ and $c_0>0$. We will say that the restricted
eigenvalue condition holds if
\[
\kappa(s,c_0):=\mathop{\min_{J\subset\{1,\ldots, D\}}}_{|J|\leq s}
\mathop{\min_{u\in\mb R^D,u\ne0}}_{\llVert u_{J^c}\rrVert _1\leq
c_0\llVert u_J\rrVert _1} \frac{\|\mb Xu\|}{\sqrt n\|u_J\|}>0.
\]
\end{definition}

Let $\Theta:= \|\frac{1} n \sum_{j=1}^n \xi_j x_j \|_{\infty}$.
The following result shows that the amount of regularization $\eps$
sufficient for recovery of $\lambda_0$ is closely related to the size
of $\Theta$.
%

\begin{theorem}
\label{th:lasso1}
Assume that the diagonal elements of the matrix $\frac{\mb X^T\mb
X}{n}$ are bounded by $1$ and
\[
\kappa\bigl(2N(\lambda_0),3\bigr)>0.
\]
On the event $\m E=  \{\eps\geq4\Theta  \}$, the following
inequality holds:
\[
\llVert \hat\lambda_\eps-\lambda_0\rrVert ^2
\leq64\eps^2\frac
{N(\lambda_0)}{\kappa^4(2N(\lambda_0),3)}.
\]
In particular, when $\xi_j\sim N(0,\sigma^2)$ and $\eps=4\sigma
t\sqrt{\frac{\log D}{n}}$,
\[
\Pr(\m E)\geq1-\frac{2}{D^{t^2-2}}.
\]
\end{theorem}

\begin{pf}
This result is similar to the statement of Theorem~7.2 in \cite
{bickel2009simultaneous}, and its proof can be obtained along the same lines.
See \cite{minsker2013geometric} for more details.
\end{pf}

Our next goal is to construct an estimator of $\lambda_0$ which admits
high confidence error bounds without restrictive assumptions on the
noise variable, such as sub-Gaussian tails.
Let $t>0$ be fixed, and set $k:=\lfloor3.5 t\rfloor+1$, $m=
\lfloor
\frac{n}{k}  \rfloor$ (as before, we will assume that $k\leq
\frac{n} 2$).
For $1\leq l \leq k$, let $G_l:=\{(l-1)m+1,\ldots,lm\}$ and
\[
\mb X_l= (x_{j_1}|\ldots|x_{j_m} )^T,
\qquad j_i=i+(l-1)m\in G_l
\]
be the $m\times D$ design matrix corresponding to the $l$th group of
design vectors $\{x_j, j\in G_l\}$.
Moreover, let $\kappa_l(s,c_0)$ be the corresponding restricted eigenvalues.

Define
\[
\hat\lambda^l_\eps:= \argmin _{\lambda\in\mb R^D} \biggl[
\frac{1}{|G_l|}\sum_{j\in G_l} \bigl(Y_j-
\lambda^T x_j \bigr)^2+\eps\| \lambda
\|_1 \biggr]
\]
and
%
%
\begin{equation}
\label{eq:median_lasso} \hat\lambda_{\eps}^\ast:=\med\bigl(\hat
\lambda^1_\eps,\ldots,\hat \lambda^{k}_\eps
\bigr),
\end{equation}
where the geometric median is taken with respect to the standard
Euclidean norm in $\mb R^D$.
The following result holds.

%
\begin{theorem}
\label{th:lasso2}
Assume that $\|x_j\|_\infty\leq M, 1\leq j\leq n$ and $\bar\kappa
(2N(\lambda_0),3):=\min_{1\leq l\leq k}\kappa_l(2N(\lambda_0),\allowbreak  3)>0$.
Then for any
\[
\eps\geq95 M\sigma\sqrt{\frac{t+2/7}{n}\log(2D)},
\]
with probability $\geq1-\mathrm{e}^{-t}$
\[
\bigl\llVert \hat\lambda_\eps^\ast-\lambda_0
\bigr\rrVert ^2\leq83\eps ^2\frac
{N(\lambda_0)}{\bar\kappa^4(2N(\lambda_0),3)}.
\]
\end{theorem}

\begin{pf}
We will first obtain a ``weak concentration'' bound from Theorem~\ref{th:lasso1} and then apply Theorem~\ref{th:main} with
$\alpha=\frac{7}{18}$ to get the result.

To this end, we need to estimate $\Theta_l:=\llVert \frac{1} m \sum_{j\in
G_l} \xi_j x_j\rrVert _{\infty}$, $l=1,\ldots,k$.
%

\begin{lemma}[(Nemirovski's inequality, Lemma~5.2.2 in \cite
{nemirovski2000topics} or Lemma~14.24 in \cite{buhlmann2011statistics})]
Assume that $D\geq3$. Then for any $l$, $1\leq l\leq k$,
\[
\mb E \Theta_l^2\leq\frac{8 \log(2D)}{m}
\frac{1}{m}\sum_{j\in
G_l} \|x_j
\|^2_{\infty} \mb E\xi_j^2.
\]
\end{lemma}

By our assumptions, $\|x_j\|_\infty\leq M$ and $\mb E\xi_j^2\leq
\sigma^2$ for all $j$, hence Chebyshev's inequality gives that for any
$1\leq l\leq k$,
\[
\Pr (\Theta_l\geq t )\leq\frac{8\log(2D)M^2\sigma
^2}{mt^2}\leq0.1
\]
whenever $t\geq4\sigma M\sqrt{\frac{k\log(2D)}{0.1 n }}$.
In particular, for $\eps\geq16\sigma M\sqrt{\frac{3.5(t+2/7)\log
(2D)}{ 0.1 n}}$, the bound of Theorem~\ref{th:lasso1} holds for $\hat
\lambda_\eps^l$ with probability $\geq1-0.1$;
note that $16\sqrt{\frac{3.5}{0.1}}\leq95$.
It remains to apply Theorem~\ref{th:main} to complete the proof.
\end{pf}

%
\begin{remark}
We stated the bounds only for the Euclidean distance $\|\hat\lambda
_\eps^\ast-\lambda_0\|$; this formulation is close to the compressed
sensing framework \cite{candes2006stable}.
If, for example, the design vectors $x,x_1,\ldots, x_n$ are i.i.d.
with some known distribution $\Pi$, one can use the median with
respect to $\|\cdot\|_{L_2(\Pi)}$ norm in the definition of $\hat
\lambda_\eps^\ast$ and obtain the bounds for the
\textit{prediction error} $\llVert \hat\lambda_\eps^\ast-\lambda
_0\rrVert ^2_{L_2(\Pi)}:=\mb E  ((\hat\lambda_\eps^\ast-\lambda_0)^T x
  )^2$.
\end{remark}

\subsection{Matrix regression with isotropic sub-Gaussian design}
\label{sec:low_rank}

In this section, we will extend some results related to recovery of
low-rank matrices from noisy linear measurements to the case of
heavy-tailed noise distribution.
Assume that the random couple $(X,Y)\in\mb R^{D\times D}\times\mb R$
is generated according to the following matrix regression model:
%
%
\begin{equation}
\label{eq:trace_reg} Y=f_\ast(X)+\xi,
\end{equation}
where $f_\ast$ is the regression function,
$X\in\mb R^{D\times D}$ is a random symmetric matrix with (unknown)
distribution
$\Pi$ and $\xi$ is a zero-mean random variable independent of $X$
with $\var(\xi)\leq\sigma^2$.
We will be mostly interested in a situation when $f_\ast$ can be well
approximated by a linear functional $ \langle A,\cdot
\rangle$, where $A$ a
symmetric matrix of small rank and $ \langle A_1,A_2 \rangle
:=\tr(A_1^T A_2)$.

The problem of low-rank matrix estimation has attracted a lot of
attention during the last several years, for example, see \cite
{candes2009exact,recht2010guaranteed} and references therein.
Recovery guarantees were later extended to allow the presence of noise.
Results in this direction can be found in \cite
{candes2011tight,rohde2011estimation,koltchinskii2011nuclear,negahban2012restricted},
to name a few.

In this section, we mainly follow the approach of \cite
{Koltchinskii2011Oracle-inequali00} (Chapter~9) which deals with an
important case of \textit{sub-Gaussian design} (also, see \cite
{candes2011tight,negahban2011estimation} for a discussion of related
problems), and use results of this work as a basis for our exposition.
Everywhere below, $\|\cdot\|_{\mathrm F}$ stands for the Frobenius
norm of
a matrix, $\|\cdot\|_\op$ -- for the operator (spectral) norm, and $\|
\cdot\|_1$ -- for the nuclear norm of a matrix.

Given $A\in\mb R^{D\times D}$, denote
$\|A\|^2_{L_2(\Pi)}:=\mb E  (\tr(A^T X)  )^2$.
Recall that a random variable $\zeta$ is called \textit{sub-Gaussian}
with parameter $\gamma^2$ if for all $s\in\mb R$,
$\mb E^{s\zeta}\leq\mathrm{e}^{s^2\gamma^2/2}$.
We will be interested in the special case when $X$ is
\textit{sub-Gaussian}, meaning that there exists $\gamma=\gamma(\Pi
)>0$ such
that for all symmetric matrices $A$, $ \langle A,X \rangle$
is a sub-Gaussian
random variable with parameter\vspace*{2pt} $\gamma^2\|A\|^2_{L_2(\Pi)}$ (in
particular, this is the case when the entries of $X$ are jointly
Gaussian, with $\gamma=1$).
Additionally, we will assume that $X$ is isotropic, so that $\|A\|
_{L_2(\Pi)}=\|A\|_{{\mathrm F}}$ for any symmetric matrix $A$.

In particular, these assumptions hold in the following important cases:
\begin{enumerate}[(b)]
\item[(a)] $X$ is symmetric and such that $\{X_{i,j}, 1\leq i\leq
j\leq D\}$ are i.i.d. centered normal random variables with $\mb
EX_{i,j}^2=\frac{1}2$, $i<j$ and $\mb EX_{i,i}^2=1$, $i=1,\ldots, D$.
\item[(b)]$X$ is symmetric and such that $X_{i,j}=\frac{1}{\sqrt
2}\eps_{i,j}$, $i<j$, $X_{i,i}=\eps_{i,i}$, $1\leq i\leq D$, where
$\eps
_{i,j}$ are i.i.d. Rademacher random variables (i.e., random signs).
\item[(c)] In a special case when all involved matrices are diagonal,
the problem becomes a version of sparse linear regression with random design.
In this case, isotropic design includes a situation when $X$ is a
random diagonal matrix $X=\operatorname{diag}(x_1,\ldots,x_D)$, where
$x_i$ are
i.i.d. standard normal or Rademacher random variables.
\end{enumerate}

%
\begin{remark}
In what follows, $C_1,C_2,\ldots$ denote the constants that may depend
on parameters of the underlying distribution (such as $\gamma$).
\end{remark}

Given $\alpha\geq1$, define $\|\zeta\|_{\psi_\alpha}:=\min
\{
r>0: \mb E\exp  (  (\frac{|\zeta|}{r}  )^{\alpha
}  )\leq2  \}$.
We will mainly use $\|\cdot\|_{\psi_\alpha}$-norms for $\alpha=1,2$.
The following elementary inequality holds: for any random variables
$\zeta_1$,~$\zeta_2$,
%
%
\begin{equation}
\label{eq:psi1-psi2} \|\zeta_1\zeta_2\|_{\psi_1}\leq\|
\zeta_1\|_{\psi_2}\|\zeta_2\| _{\psi_2}.
\end{equation}
It is easy to see that $\|\zeta\|_{\psi_2}<\infty$ for any
sub-Gaussian random variable $\zeta$.
It follows from Proposition~9.1 in \cite
{Koltchinskii2011Oracle-inequali00} that there exists $C(\gamma)>0$
such that for any sub-Gaussian isotropic matrix~$X$,
%
%
\begin{equation}
\label{eq:psi2} \bigl \|\|X\|_\op \bigr \|_{\psi_2}\leq C(\gamma)\sqrt D.
\end{equation}
We will also need the following useful inequality:
for any $p\geq1$,
%
%
\begin{equation}
\label{eq:equiv} \mb E^{1/p}\bigl | \langle A,X \rangle\bigr |^{p}\leq
C_{p,\gamma}\|A\| _{L_2(\Pi)}.
\end{equation}
The proofs of the facts mentioned above can be found in \cite
{Koltchinskii2011Oracle-inequali00}.

Let $(X_1,Y_1),\ldots,(X_n,Y_n)$ be i.i.d. observations with the same
distribution as $(X,Y)$.
We are mainly interested in the case when $D< n\ll D^2$.
In this situation, it is impossible to estimate $A_0$ consistently in
general, however, if $A_0$ is low-rank (or approximately low-rank),
then the solution of the following optimization problem provides a good
approximation to $A_0$:
%
%
\begin{equation}
\label{eq:nuc_norm} \hat A_\eps:=\argmin_{A\in\mb L} \Biggl[
\frac{1} n\sum_{j=1}^n
\bigl(Y_j - \langle A,X_j \rangle \bigr)^2+
\eps\|A\|_1 \Biggr].
\end{equation}
Here, $\mb L$ is a bounded, closed, convex subset of a set of all
$D\times D$ symmetric matrices.
%

\begin{remark}
All results of this subsection extend to the case of unbounded $\mb L$
and non-isotropic sub-Gaussian design.
However, our assumptions still cover important examples and yield less
technical statements; see Theorem~9.3 in \cite
{Koltchinskii2011Oracle-inequali00} for details on the general case.
Results for the arbitrary rectangular matrices follow from the special
case discussed here, see the remark on page 202 of \cite
{Koltchinskii2011Oracle-inequali00}.
\end{remark}

We proceed by recalling the performance guarantees for $\hat A_\eps$.
Let $R_\mb L:=\sup_{A\in\mb L}\|A\|_1$, and define
\[
\Theta:=\frac{1} n \sum_{j=1}^n
\xi_j X_j.
\]
%

\begin{theorem}[(Theorem~9.4 in \cite{Koltchinskii2011Oracle-inequali00})]
\label{th:nuclear1}
There exist constants $c$, $C$ with the following property:
let $\kappa:=\log\log_2 (D R_\mb L)$, and assume that $t\geq1$ is
such that $t_{n,D}\leq cn$, where $t_{n,D}:=(t+\kappa)\log n+\log(2D)$.
Define the event $\m E:=  \{\eps\geq2\|\Theta\|_\op  \}$.
The following bound holds with probability $\geq\Pr  (\m
E  )-\mathrm{e}^{-t}$:
%
%
\begin{equation}
\label{eq:comp1} \|\hat A_\eps-A_0\|^2_{\mathrm F}
\leq\inf_{A\in\mb L} \biggl[2\|A-A_0\|^2_{\mathrm F}+C
\biggl(\eps^2 \rank(A)+R^2_\mb L
\frac
{Dt_{n,D}}{n}+\frac{1}{n} \biggr) \biggr].
\end{equation}
\end{theorem}

Constant 2 in front of $\|A-A_0\|^2_{\mathrm F}$ can be replaced by
$(1+\nu
)$ for any $\nu>0$ if $C$ is replaced by $C/\nu$.

%
\begin{assumption}
\label{subgauss}
The noise variable $\xi$ is such that $\|\xi\|_{\psi_2}<\infty$.
\end{assumption}

If Assumption~\ref{subgauss} is satisfied, then, whenever
%
%
\begin{equation}
\label{eq:eps1} \eps\geq\bar C(\gamma)\sqrt{\frac{D} n} \biggl(\sigma
\sqrt{t+\log(2D)}\vee\|\xi\|_{\psi_2}\log \biggl(2\vee \frac{\|\xi\|_{\psi_2}}{\sigma}
\biggr)\frac{t+\log(2D)}{\sqrt
n} \biggr)
\end{equation}
we have that $\Pr(\m E)\geq1-\mathrm{e}^{-t}$ (hence, (\ref{eq:comp1}) holds
with probability $1-2\mathrm{e}^{-t}$).
This follows from the following variant of the non-commutative
Bernstein's inequality.
%

\begin{theorem}[(Theorem~2.7 in \cite{Koltchinskii2011Oracle-inequali00})]
\label{th:bernstein}
Let $Y_1,\ldots, Y_n\in\mb R^{D\times D}$ be symmetric independent
random matrices such that $\mb EY_j=0$ and
\[
\max_{1\leq j\leq n} \bigl( \bigl \|\|Y_j\|_\op
\bigr \|_{\psi_1}\vee 2\mb E^{1/2}\|Y_j\|^2_\op
\bigr)\leq U.
\]
Let
\[
\Psi^2\geq\frac{1} n\Biggl\llVert \sum
_{i=1}^n \mb EY_i^2\Biggr
\rrVert _\op.
\]
Then, for all $t>0$, with probability $\geq1-\mathrm{e}^{-t}$
\[
\Biggl\llVert \frac{1}{\sqrt n}\sum_{j=1}^n
Y_j\Biggr\rrVert _\op\leq\bar C_1\max
\biggl(\Psi\sqrt{t+\log(2D)},U\log \biggl(\frac{U}{\Psi} \biggr)
\frac
{t+\log
(2D)}{\sqrt n} \biggr),
\]
where $\bar C_1>0$ is an absolute constant.
\end{theorem}

Indeed, recall that $\Theta=\frac{1} n\sum_{j=1}^n \xi_j X_j$, and
apply Theorem~\ref{th:bernstein} to $Y_j:=\xi_j X_j$, noting that by
(\ref{eq:psi1-psi2}), (\ref{eq:psi2})
\begin{eqnarray*}
\bigl\llVert \mb E\xi^2 X^2\bigr\rrVert
_\op&\leq&\sigma^2 \mb E\|X\|_\op
^2\leq C_2 \sigma^2 D,
\\
\bigl \|\|\xi X\|_\op \bigr \|_{\psi_1}&\leq&\|\xi_1
\|_{\psi_2} \bigl \| \|X\|_\op \bigr \|_{\psi_2}\leq C(\gamma)\|\xi
\|_{\psi_2} \sqrt{D}.
\end{eqnarray*}
It implies that with probability $\geq1-\mathrm{e}^{-t}$,
%
%
\begin{equation}
\label{eq:d2} \Biggl\|\frac{1} n\sum_{j=1}^n
\xi_j X_j \Biggr\|_{\op}\leq C_3\sqrt {
\frac{D} n} \biggl(\sigma\sqrt{t+\log(2D)}\vee\|\xi\|_{\psi_2}\log
\biggl(2\vee\frac{\|\xi\|_{\psi_2}}{\sigma
} \biggr)\frac{t+\log(2D)}{\sqrt n} \biggr),
\end{equation}
where $C_2$, $C_3$ depend only on $\gamma$, hence giving the result.

As we mentioned above, our goal is to construct the estimator of $A_0$
which admits bounds in flavor of Theorem~\ref{th:nuclear1} that hold
with high probability under a much weaker assumption on the tail of the
noise variable $\xi$.

To achieve this goal, we follow the same pattern as before.
Let $t\geq1$ be fixed, let $k:=\lfloor t\rfloor+1, m=  \lfloor
\frac
{n}{k}  \rfloor$, and assume that $k\leq\frac{n} 2$.
Divide the data $  \{(X_j,Y_j)  \}_{j=1}^n$ into $k$
disjoint groups
$G_1,\ldots, G_k$ of size $m$ each, and define
\[
\hat A^l_\eps:=\argmin_{A\in\mb L} \biggl[
\frac{1}{|G_l|}\sum_{j\in G_l} \bigl(Y_j -
\langle A,X_j \rangle \bigr)^2+\eps \|A\|_1
\biggr]
\]
and
\[
\hat A_{\eps}^\ast=\hat A_{\eps}(t):=\med\bigl(\hat
A^1_\eps,\ldots ,\hat A^{k}_\eps
\bigr),
\]
where the geometric median is evaluated with respect to the Frobenius norm.
%

\begin{assumption}
\label{noise_moment}
\[
\|\xi\|_{2,1}:=\int_0^\infty\sqrt{\Pr (|
\xi|>x )}\,\mathrm{d}x<\infty.
\]
\end{assumption}

In particular, $\|\xi\|_{2,1}<\infty$ if $\mb E|\xi|^{2+\delta
}<\infty$ for some $\delta>0$, which is a mild requirement compared
to Assumption~\ref{subgauss}.
Finally, given $\alpha\in(0,1/2)$, it will be convenient to define
\[
p^\ast=p^\ast(\alpha):=\max\bigl\{p\in(0,\alpha): \psi(
\alpha;p)\geq 1\bigr\}.
\]

%
\begin{theorem}
\label{th:nuclear2}
Suppose that Assumption~\ref{noise_moment} is satisfied.
For any $\alpha\in(0,1/2)$, there exist constants $c$, $C$, $B$ with the
following properties:
let $\kappa:=\log\log_2 (D R_\mb L)$, $s_{n,t,D}:=(\log(2/p^\ast
(\alpha))+\kappa)\log(n/t)+\log(2D)$, and assume that
$s_{n,t,D}\leq c(n/t)$.
Then for all
\[
\eps\geq\frac{B}{p^\ast(\alpha)}\|\xi\|_{2,1}\sqrt{\frac
{Dt}{n}}
\log(2D),
\]
with probability $\geq1-2\mathrm{e}^{-t}$
%
%
\begin{equation}
\label{eq:d6} \bigl \|\hat A_\eps^\ast-A_0
\bigr \|^2_{\mathrm F}\leq C_\alpha\inf_{A\in\mb L}
\biggl[2\|A-A_0\|^2_{\mathrm F}+C \biggl(
\eps^2 \rank(A)+R^2_\mb L s_{n,t,D}
\frac{Dt}{n}+\frac{t}{n} \biggr) \biggr],
\end{equation}
where $C_\alpha$ is defined by (\ref{eq:c_alpha}).
\end{theorem}

\begin{pf}
We will start by deriving a ``weak concentration'' bound from Theorem~\ref{th:nuclear1}.
To this end, we need to estimate
\[
\mb E\|\Theta_l\|_\op:=\mb E \biggl\|\frac{1}{|G_l|}\sum
_{j\in
G_l}\xi_j X_j
\biggr\|_\op,\qquad l=1,\ldots, k.
\]
The following result is a direct consequence of the so-called
\textit{multiplier inequality} (Lemma~2.9.1 in \cite
{Vaart1996Weak-convergenc00}).
%

\begin{lemma}
Let $\eps_1,\ldots,\eps_m$ be i.i.d. Rademacher random variables
independent of $X_1,\ldots,X_m$.
Then
%
%
\begin{equation}
\label{d3} \mb E \Biggl\|\frac{1}{m}\sum_{j=1}^m
\xi_j X_j \Biggr\|_\op\leq \frac{2\sqrt2\|\xi\|_{2,1}}{\sqrt{m}}\max
_{1\leq i\leq m} \mb E \Biggl\|\frac{1}{\sqrt i}\sum
_{j=1}^i \eps_j X_j
\Biggr\|_\op.
\end{equation}
\end{lemma}

To estimate $\mb E \|\frac{1}{\sqrt i}\sum_{j=1}^i \xi_j X_j
\|_\op$, we use the formula $\mb E|\eta|=\int_0^\infty\Pr(|\eta
|\geq t)\,\mathrm{d}t$ and the tail bound of Theorem~\ref{th:bernstein}, which
implies (in a way similar to (\ref{eq:d2})) that with probability
$\geq1-\mathrm{e}^{-t}$
%
%
\begin{equation}
\label{eq:d4} \Biggl\|\frac{1}{\sqrt i} \sum_{j=1}^i
\eps_j X_j\Biggr \|_{\op}\leq C_4
\sqrt{D} \biggl(\sqrt{t+\log(2D)}\vee\frac{t+\log(2D)}{\sqrt i} \biggr),
\end{equation}
hence for any $1\leq i\leq m$
\[
\mb E \Biggl\|\frac{1}{\sqrt i}\sum_{j=1}^i
\eps_j X_j \Biggr\|_\op\leq C_5
\sqrt{D} \biggl(\sqrt{\log(2D)}+\frac{\log(2D)}{\sqrt i} \biggr),
\]
and (\ref{d3}) yields
\[
\mb E \Biggl\|\frac{1}{m}\sum_{j=1}^m
\xi_j X_j \Biggr\|_\op\leq C_6\| \xi
\|_{2,1}\sqrt{\frac{D} m}\log(2D).
\]
Next, it follows from Chebyshev's inequality that for any $1\leq l\leq
k$, with probability $\geq1-\frac{p^\ast(\alpha)}{2}$
\[
2\llVert \Theta_l\rrVert _\op\leq\frac{4C_6}{p^\ast(\alpha)}\|
\xi\| _{2,1}\sqrt{\frac{D} m}\log(2D).
\]
Hence, if $\alpha\in(0,1/2)$ and
\[
\eps\geq\frac{4C_6}{p^\ast(\alpha)}\|\xi\|_{2,1}\sqrt{\frac{D}m }
\log(2D),
\]
the inequality of Theorem~\ref{th:nuclear1} (with the confidence
parameter equal to $\log(2/p^\ast(\alpha))$) applied to the estimator
$\hat A^l_\eps$ gives that with probability $\geq1-p^\ast(\alpha)$
\[
\bigl \|\hat A_\eps^\ast-A_0\bigr \|^2_{\mathrm F}
\leq\inf_{A\in\mb L} \biggl[2\|A-A_0\|^2_{\mathrm F}+C
\biggl(\eps^2 \rank(A)+R^2_\mb L
s_{m,D} \frac
{D}{m}+\frac{1}{m} \biggr) \biggr].
\]
The claim (\ref{eq:d6}) now follows from Theorem~\ref{th:main}.
\end{pf}

\section{Numerical evaluation of the geometric median and simulation results}

In this section, we briefly discuss computational aspects of our method
in $\mb R^D$ equipped with the standard Euclidean norm $\|\cdot\|$,
and present results of numerical simulation.

\subsection{Overview of some numerical algorithms}

As was mentioned in the \hyperref[sec:intro]{introduction}, the
function $F(z):=\sum_{j=1}^k
\|z-x_j\|$ is convex, moreover, its minimum is unique unless
$\{x_1,\ldots,x_k\}$ are on the same line.

One of the computationally efficient ways to approximate $\operatorname{arg\,min}
_{z\in\mb R^D} F(z)$ is the famous \textit{Weiszfeld's algorithm}
\cite{Weiszfeld1936Sur-un-probleme00}: starting from some $z_0$ in the
affine hull of $\{x_1,\ldots,x_k\}$, iterate
%
%
\begin{equation}
\label{weizsfeld} z_{m+1}=\sum_{j=1}^k
\alpha^{(j)}_{m+1} x_j,
\end{equation}
where $\alpha^{(j)}_{m+1}=\frac{\|x_j-z_m\|^{-1}}{\sum_{j=1}^k \|
x_j-z_m\|^{-1}}$.
Kuhn proved \cite{kuhn1973note} that Weiszfeld's algorithm
converges to the geometric median for all but countably many initial
points (additionally, his result states that $z_m$ converges to the
geometric median if none of $z_m$ belong to $\{x_1,\ldots,x_k\}$).
It is straightforward to check that (\ref{weizsfeld}) is actually a
gradient descent scheme: indeed, it is equivalent to
\[
z_{m+1}=z_m-\beta_{m+1}g_{m+1},
\]
where $\beta_{m+1}=\frac{1}{\sum_{j=1}^k\|x_j-z_m\|^{-1}}$ and
$g_{m+1}=\sum_{j=1}^k \frac{z_m-x_j}{\|z_m-x_j\|}$ is the gradient of
$F$ (we assume that $z_m\notin\{x_1,\ldots,x_k\}$).

Ostresh \cite{ostresh1978convergence} proposed a method which
avoids the possibility of hitting one of the vertices $\{x_1,\ldots
,x_k\}$
by considering the following descent scheme: starting with some $z_0$
in the affine hull of $\{x_1,\ldots,x_k\}$, let
\[
z_{m+1}=z_m-\zeta\tilde\beta_{m+1} \tilde
g_{m+1},
\]
where $\zeta\in[1,2]$, $\tilde g_{m+1}$ is the properly defined
``generalized'' gradient (see \cite{ostresh1978convergence} for
details), and
$\tilde\beta_{m+1}=\frac{1}{\sum_{j: x_j\ne z_m}\|x_j-z_m\|^{-1}}$.
It is shown that $z_{m}$ converges to the geometric median whenever it
is unique.
Further improved modifications of original Weiszfeld's method can be
found in \cite{vardi2000multivariate}.

For other approaches to fast numerical evaluation of the geometric
median, see \cite
{overton1983quadratically,chandrasekaran1990algebraic,bose2003fast,cardot2013efficient}
and references therein.

\subsection{Simulation results}

\subsubsection{Principal component analysis}

Data points $X_1,\ldots,X_{156}$ were sampled from the distribution on
$\mb R^{120}$ such that
$X_1\overset{d}{=}AY$, where the coordinates of $Y$ are independent
random variables with density $p(y)=\frac{3y^2}{2(1+|y|^3)^2}$ and\vspace*{2pt}
$A$ is a full-rank diagonal matrix with 5 ``large'' eigenvalues $\{
5^{1/2},6^{1/2},7^{1/2},8^{1/2},9^{1/2}\}$
while the remaining diagonal elements are equal to $\frac{1}{\sqrt{120}}$.
Additionally, the data set contained $4$ ``outliers'' $Z_1,\ldots,Z_4$
generated from the uniform distribution on $[-20,20]^{120}$ and
independent of $X_i$'s.

In this case, the usual sample covariance matrix does not provide any
useful information about the principal components.
However, in most cases our method gave reasonable approximation to the truth.
We used the estimator described in Section~\ref{sec:pca} with\vspace*{1pt} the
number of groups $k=10$ containing $16$ observations each.
The error was measured by the spectral norm $\llVert \widehat\proj
_5-\proj_5\rrVert _\op$, where $\widehat\proj_5$ is a projector
on the
eigenvectors corresponding to $5$ largest eigenvalues of the estimator.
Figures~\ref{fig:02}, \ref{fig:03} show the histograms of the errors
evaluated over $100$ runs of the simulation.
Figure~\ref{fig:04} shows performance of a ``thresholded geometric
median'' estimator which is defined in Section~\ref{sec:final} below.

%
\begin{figure}[b]

\includegraphics{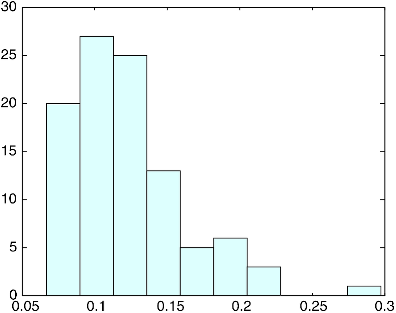}

\caption{Error of the ``geometric median'' estimator (\protect\ref{eq:median_cov}).}
\label{fig:02}
\end{figure}

%
\begin{figure}[t]

\includegraphics{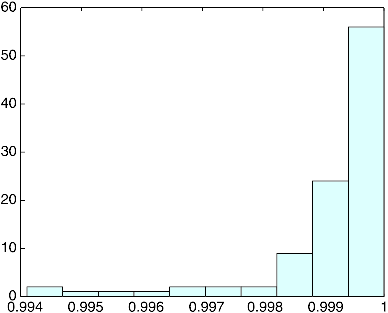}

\caption{Error of the sample covariance estimator.}
\label{fig:03}
\end{figure}
%

%
\begin{figure}[b]

\includegraphics{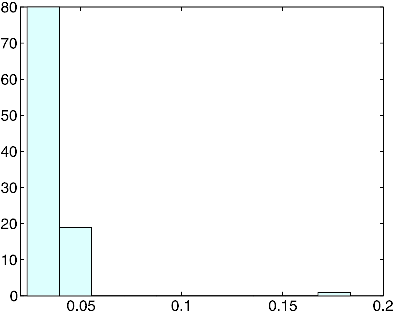}

\caption{Error of the ``thresholded geometric median'' estimator
(\protect\ref{eq:threshold}), $\nu=0.5$.}
\label{fig:04}
\end{figure}

\subsubsection{High-dimensional sparse linear regression}

The following model was used for simulation:
\[
Y_j=\lambda_0^T x_j+
\xi_j,\qquad j=1,\ldots,300,
\]
where $\lambda_0\in\mb R^{1000}$ is a vector with $10$ non-zero
entries sampled from the uniform distribution on $[-15,15]$, and
$x_j\in\mb R^{1000}$, $j=1,\ldots,300$, are generated according to the
normal distribution $N(0,I_{1000})$.
Noise $\xi_j$ was sampled from the mixture
\[
\xi_j= %
\cases{ \xi_{1,j} &\quad$\mbox{with
probability } 1-1/500$,
\cr
\xi_{2,j} &\quad$\mbox{with probability }
1/500$,} %
\]
where $\xi_{1,j}\sim N(0,1/8)$ and $\xi_{2,j}$ takes values $\pm
\frac{250}{\sqrt2}$ with probability $1/2$ each.
All parameters $\lambda_0$, $x_j$, $\xi_j$, $j=1,\ldots,300$, were sampled
independently.
Error of the estimator $\hat\lambda$ was measured by the ratio $\frac
{\|\hat\lambda-\lambda_0\|}{\|\lambda_0\|}$.
Size of the regularization parameter $\eps$ was chosen based on
$4$-fold cross validation.
On each stage of the simulation, we evaluated the usual Lasso estimator
(\ref{eq:lasso}) and the ``median Lasso'' estimator
(\ref{eq:median_lasso}) based on partitioning the observations into $4$
groups of size $75$ each.
Figures~\ref{fig:05} and~\ref{fig:06} show the histograms of the
errors over $50$ runs of the simulation.
Note that the maximal error of the ``median Lasso'' is $0.055$ while
the error of the usual Lasso exceeded $0.15$ in $18$ out of $50$ cases.

%
\begin{figure}[t]

\includegraphics{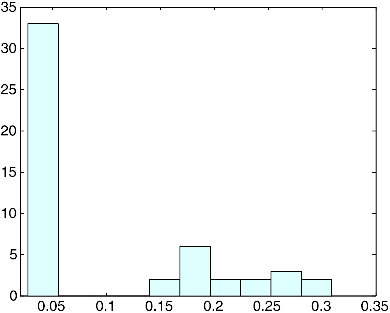}

\caption{Error of the standard Lasso estimator (\protect\ref{eq:lasso}).}
\label{fig:05}
\end{figure}

%
\begin{figure}[b]

\includegraphics{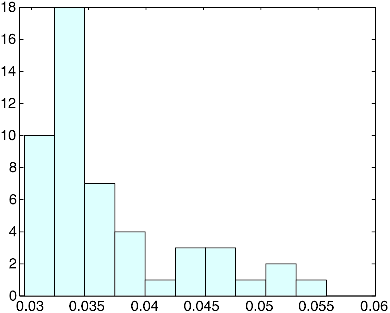}

\caption{Error of the ``median Lasso'' estimator (\protect\ref{eq:median_lasso}).}
\label{fig:06}\vspace*{-4pt}
\end{figure}

\section{Final remarks}
\label{sec:final}

Let $\hat\alpha_1,\ldots,\hat\alpha_k\geq0$, $\sum_{j=1}^k
\alpha_j=1$ be the coefficients such that
$
\hat\mu=\sum_{j=1}^k \hat\alpha_k \mu_k$
is the geometric median of a collection of estimators $\{\mu_1,\ldots
,\mu_k\}$.
Our numerical experiments reveal that performance of $\hat\mu$ can be
significantly improved by setting the coefficients below a certain
threshold level
$\nu$ to $0$, that is,\vspace*{-2pt}
%
%
\begin{eqnarray}
\label{eq:threshold} \tilde\alpha_j&:=&\frac{\hat\alpha_j I  \{\hat\alpha_j\geq
\nu
/k  \}}{\sum_{i=1}^k \hat\alpha_j I  \{\hat\alpha_j\geq
\nu/k
\}},
\nonumber
\\[-9pt]
\\[-9pt]
\tilde\mu&:=&\sum_{j=1}^k \tilde
\alpha_j\mu_j.
\nonumber
\end{eqnarray}
An interesting problem that we plan to address in subsequent work is
the possibility of adaptive choice of the threshold parameter.

Examples presented above cover only a small area on the map of possible
applications.
For instance, it would be interesting to obtain an estimator in the
low-rank matrix completion framework \cite
{candes2009exact,koltchinskii2011nuclear} that admits strong
performance guarantees for the heavy-tailed noise model.
Results obtained in Section~\ref{sec:low_rank} for the matrix
regression problem do not seem to yield a straightforward solution in
this case.
Another promising direction is related to design of robust techniques
for Bayesian inference and evaluation of the geometric median in the
space of probability measures. We plan to address these questions in
the future work.\vspace*{-2pt}

\begin{appendix}\label{part2}
\section*{Appendix: Proof of Lemma~\texorpdfstring{\protect\ref{lem:median}}{2.1}, part (b)}

Once again, assume that the claim does not hold and $\|x_i-z\|\leq r$,
$i=1,\ldots,  \lfloor(1-\alpha)k  \rfloor+1$.

We will need the following general description of the subdifferential
of a norm $\|\cdot\|$ in a Banach space $\mb X$ (see, e.g., \cite
{ioffe1974theory}):\vspace*{-2pt}
\[
\partial\|x\|= %
\cases{ \bigl\{x^\ast\in\mathfrak
X^\ast: \bigl \|x^\ast\bigr \|_\ast= 1, x^\ast (x)=\|
x\| \bigr\},&\quad $x\ne0$,
\cr
\bigl\{x^\ast\in\mathfrak
X^\ast: \bigl \|x^\ast\bigr \|_\ast\leq1 \bigr\} ,&\quad $x=0$,}
\]
where $\mb X^*$ is the dual space with norm $\|\cdot\|_\ast$.

For $x,u\in\mb X$, let
\[
D\bigl(\|x\|;u\bigr)=\lim_{t\searrow0}\frac{\|x+tu\|-\|x\|}{t}
\]
be the directional derivative of $\|\cdot\|$ at the point $x$ in
direction $u$.
We need the following useful fact from convex analysis:
%

\begin{lemma}
\label{grad}
There exists $g^\ast:=g^\ast_{x,u}\in\partial\|x\|$ such that $D(\|
x\|;u)= \langle g^\ast,u \rangle$, where
$
 \langle g^\ast,u \rangle:=g^\ast(u)$.
\end{lemma}

\begin{pf}
In follows from the results of Chapter~4 in \cite{ioffe1974theory} that
$D(\|x\|;u)$ is a continuous convex function of $u$ (for fixed $x$),
hence its subdifferential is non-empty.
Let $\tilde g\in\partial D(\|x\|;u)$. Then for all $s>0, v\in\mb X$
\[
sD\bigl(\|x\|;v\bigr)=D\bigl(\|x\|;sv\bigr)\geq D\bigl(\|x\|;u\bigr)+ \langle\tilde g,sv-u \rangle.
\]
Letting $s\to\infty$, we get $D(\|x\|;v)\geq \langle\tilde
g,v \rangle$,
hence $\tilde g\in\partial\|x\|$.
Taking $s=0$, we get $D(\|x\|;u)\leq \langle\tilde g,u
\rangle$, hence
$g^\ast:=\tilde g$ satisfies the requirement.
\end{pf}

Using Lemma~\ref{grad}, it is easy to see that there exist $g^\ast
_j\in\partial\|x_j-x_\ast\|$, $j=1,\ldots, k$ such that
\[
DF \biggl(x_\ast;\frac{z-x_\ast}{\|z-x_\ast\|} \biggr)=-\sum
_{j:
x_j\ne
x_\ast} \frac{ \langle g^\ast_j,z-x_\ast \rangle}{\|
z-x_\ast\|}+\sum_{j=1}^k
I\{ x_j=x_\ast\}.
\]
Moreover, for any $u$, $DF(x_\ast;z-x_\ast)\geq0$ by the definition
of $x_\ast$.
Note that for any $j$
%
%
\setcounter{equation}{0}
\begin{equation}
\label{eq:app1} \frac{ \langle g^\ast_j,z-x_\ast \rangle}{{\|z-x_\ast\|
}} =\frac{ \langle g^\ast_j,x_j-x_\ast \rangle+
\langle g^\ast_j,z-x_j \rangle}{{\|z-x_\ast\|}}.
\end{equation}
By the definition of $g^\ast_j$ and triangle inequality,
\[
\bigl\langle g^\ast_j,x_j-x_\ast
\bigr\rangle=\|x_j-x_\ast\|\geq \|z-x_\ast\|-\|
z-x_j\|
\]
and, since $\|g^\ast_j\|_\ast\leq1$,
\[
\bigl\langle g^\ast_j,z-x_j \bigr\rangle
\geq-\|z-x_j\|.
\]
Substituting this in (\ref{eq:app1}), we get
\[
\frac{ \langle g^\ast_j,z-x_\ast \rangle}{{\|z-x_\ast\|
}}\geq1-2\frac{\|
z-x_j\|}{{\|z-x_\ast\|}}> 1-\frac{2}{C_\alpha},
\]
hence
\[
DF \biggl(x_\ast;\frac{z-x_\ast}{\|z-x_\ast\|} \biggr)< -(1-\alpha )k \biggl(1-
\frac{2}{C_\alpha} \biggr)+\alpha k\leq0
\]
whenever $C_\alpha\geq\frac{2(1-\alpha)}{1-2\alpha}$.
\end{appendix}

\section*{Acknowledgements}

S. Minsker was supported by grants NSF DMS-0847388, NSF CCF-0808847,
and R01-ES-017436 from the National Institute of Environmental Health
Sciences (NIEHS) of the National Institutes of Health (NIH).

I want to thank Anirban Bhattacharya, David Dunson, the anonymous
Referees and the Area Editor for their valuable comments and
suggestions, and Philippe Rigollet for pointing out several missing references.


%

\printhistory
\end{document}